\documentclass[12pt]{amsart}
\usepackage{amssymb}
\usepackage{amsmath}
\usepackage[active]{srcltx}
\usepackage{t1enc}
\usepackage[latin2]{inputenc}
\usepackage{verbatim}
\usepackage{amsmath,amsfonts,amssymb,amsthm}
\usepackage[mathcal]{eucal}
\usepackage{enumerate}
\usepackage[centertags]{amsmath}
\usepackage{graphics}

\newtheorem{theorem}{Theorem}

\newtheorem{corollary}{Corollary}

\newtheorem{Proposition}{Proposition}

\begin{document}
\author{D. Baramidze}
\title[$T$ means]{Sharp $\left( H_{p},L_{p}\right) $ and $\left( H_{p},\text{weak}-L_{p}\right) $ type
inequalities of weighted maximal operators of $T$ means with respect to Vilenkin systems}
\address{D. Baramidze, The University of Georgia, School of science and
	technology, 77a Merab Kostava St, Tbilisi 0128, Georgia and Department of
	Computer Science and Computational Engineering, UiT - The Arctic University
	of Norway, P.O. Box 385, N-8505, Narvik, Norway.}
\email{davit.baramidze@ug.edu.ge }
\thanks{The research was supported by Shota Rustaveli National Science Foundation grant no. PHDF-21-1702.}
\date{}
\maketitle

\begin{abstract}
We  discuss  $\left( H_{p},L_{p}\right) $ and $\left( H_{p},\text{weak}-L_{p}\right) $
type inequalities of weighted maximal operators of  $T$ means with respect to the Vilenkin systems with monotone coefficients, considered in \cite{tut4} and prove that these results are the best possible in a special sense.  As applications, both some well-known and new results are pointed out.
\end{abstract}



\section{Introduction}

The definitions and notations used in this introduction can be
found in our next Section. 

It is well-known that Vilenkin systems do not form bases in the space $L_{1}$. Moreover, there is a function in the Hardy space $H_p$, such that the
partial sums of $f$ are not bounded in $L_p$-norm, for $0<p\leq 1.$ Approximation properties of Vilenkin-Fourier series with respect to one- and two-dimensional cases can be found in Persson,  Tephnadze and Wall \cite{PTW2},  Simon \cite{Si4}, Blahota \cite{b} and Gát \cite{Ga1}, Tephnadze \cite{tep7,tep8,tep9}, Tutberidze \cite{tut1} (see also \cite{MST}). 
In the one-dimensional case the weak (1,1)-type inequality for the maximal operator of Fejér means 
$$
\sigma ^{\ast }f:=\sup_{n\in \mathbb{N}}\left\vert \sigma _{n}f\right\vert
$$
can be found in Schipp \cite{Sc} for Walsh series and in Pál, Simon \cite{PS}
for bounded Vilenkin series. Fujji \cite{Fu} and Simon \cite{Si2} verified
that $\sigma ^{\ast }$ is bounded from $H_{1}$ to $L_{1}$. Weisz \cite{We2}
generalized this result and proved boundedness of $\ \sigma ^{\ast }$ \ from
the martingale space $H_{p}$ to the space $L_{p},$ for $p>1/2$. Simon \cite%
{Si1} gave a counterexample, which shows that boundedness does not hold for $%
0<p<1/2.$ A counterexample for $p=1/2$ was given by Goginava \cite{Go} (see
also Tephnadze \cite{tep1}). Moreover, Weisz \cite{We4} proved that the maximal
operator of the Fejér means $\sigma ^{\ast }$ is bounded from the Hardy
space $H_{1/2}$ to the space $weak-L_{1/2}$. In \cite{tep2} and \cite{tep3} 
the following result  was proved:

\textbf{Theorem T1:} Let $0<p\leq 1/2.$ Then the weighted maximal operator of Fejér means $\widetilde{\sigma}_p^{\ast }$ defined by
\begin{equation*}
\widetilde{\sigma}_p^{\ast }f:=\sup_{n\in \mathbb{N}_+}\frac{\left\vert\sigma_n f \right\vert}{\left(n+1\right)^{1/p-2}\log ^{2\left[1/2+p\right]} \left( n+1\right) }
\end{equation*}
is bounded from the martingale Hardy space $H_{p}$ to the Lebesgue space $L_{p}$. 

Moreover, the rate of the weights 
$\left\{ 1/\left( n+1\right)^{1/p-2}\log ^{2\left[ p+1/2\right] }\left( n+1\right) \right\}_{n=1}^{\infty }$ 
in $n$-th Fejér mean was given exactly. 

In \cite{tep6} (see also \cite{BPTW} and \cite{lptt}) it was proved that the maximal operator of Riesz means 
$$R^{\ast}f:=\sup_{n\in\mathbb{N}}\left\vert R_{n}f\right\vert$$  
is bounded from the Hardy space $H_{1/2}$ to the space $weak-L_{1/2}$ and is not bounded from $H_{p}$ to the space $L_{p},$ for $0<p\leq 1/2.$ There was also proved that Riesz summability has better properties than Fejér means. In particular, the following weighted maximal operators  $$\frac{\log n \vert R_n f \vert }{\left( n+1\right) ^{1/p-2}\log ^{2\left[ 1/2+p\right] }\left( n+1\right) }$$ are bounded from $H_{p}$ to the space $L_{p},$ for $0<p\leq 1/2$ and the rate of weights are sharp. 

Similar results with respect to Walsh-Kachmarz systems were obtained in \cite{gn} for $p=1/2$ and in \cite{tep4} for $0<p<1/2.$ Approximation properties of Fejér means  with respect to Vilenkin and Kaczmarz systems can be found in Tephnadze \cite{tep5}, Tutberidze \cite{tut2}, Persson, Tephnadze and Tutberidze \cite{PTT1}, Blahota and Tephnadze \cite{bt1} and Persson, Tephnadze, Tutberidze and Wall \cite{PTTW1}, Gogolashvili,  Nagy and Tephnadze \cite{gnt} and Persson, Tephnadze and Wall \cite{PTW3}.

Móricz and Siddiqi \cite{Mor} investigated the approximation properties of
some special Nörlund means of Walsh-Fourier series of $L_{p}$ function in
norm. In the two-dimensional case approximation properties of Nörlund
means were considered by Nagy \cite{nag,n,nagy} (see also \cite{NT1,NT2,NT3,NT4}). In \cite{ptw} (see also \cite{BT2}, \cite{MPT} and \cite{PTW3}) it was proved that the maximal operators of Nörlund means $t^{\ast }$ defined by
$$t^{\ast }f:=\sup_{n\in \mathbb{N} }\left\vert t_{n}f\right\vert,$$ 
either with non-decreasing coefficients, or non-increasing coefficients, satisfying the condition 
\begin{equation} \label{cond0}
\frac{1}{Q_{n}}=O\left(\frac{1}{n}\right),\text{ \ \ as \ \ }
n\rightarrow \infty 
\end{equation}
 are bounded from the Hardy space $H_{1/2}$
to the space $weak-L_{1/2}$ and are not bounded from
the Hardy space $H_{p}$ to the space $L_{p},$ when $0<p\leq 1/2.$

In \cite{tut3} was proved that the maximal operators  $T^{\ast }$ of $T$ means  defined by
$$T^{\ast }f:=\sup_{n\in \mathbb{N}}\left\vert T_{n}f\right\vert$$ 
either with non-increasing coefficients, or non-decreasing sequence satisfying condition 
\begin{equation}
\frac{q_{n-1}}{Q_{n}}=O\left( \frac{1}{n}\right) ,\text{ \ \ as \ \ }\
n\rightarrow \infty, \label{fn011}
\end{equation}%
are bounded from the Hardy space $H_{1/2}$
to the space $weak-L_{1/2}$. Moreover, there exists a martingale and such $T $ means for which boundedness from
the Hardy space $H_{p}$ to the space $L_{p}$ does not hold when $0<p\leq 1/2.$

In \cite{tut4} (see also  \cite{GT1,GT2}) was proved that if $T$ is either with non-increasing coefficients, or non-decreasing sequence satisfying condition \eqref{fn011} that  the weighted maximal operator of $T$ means $\widetilde{T}_p^{\ast }$ defined by
\begin{equation}\label{conddato0}
\widetilde{T}_p^{\ast }f:=\sup_{n\in \mathbb{N}_+}\frac{\left\vert T_n f \right\vert}{\left(n+1\right)^{1/p-2}\log ^{2\left[1/2+p\right]} \left( n+1\right) }
\end{equation}
is bounded from the martingale Hardy space $H_{p}$ to the Lebesgue space $L_{p}$.

Some general means related to $T$ means was investigated by Blahota and Nagy \cite{BN} (see also \cite{BNT}).

In this paper we  discuss  $\left( H_{p},L_{p}\right) $ and $\left( H_{p},\text{weak}-L_{p}\right) $
type inequalities of weighted maximal operators of  $T$ means with respect to the Vilenkin systems with monotone coefficients, considered in \cite{tut4} and prove that the rate of the weights in \eqref{conddato0} are the best possible in a special sense.  As applications, both some well-known and new results are pointed out.

This paper is organized as follows: In order not to disturb our discussions
later on some definitions and notations are presented in Section 2. The main
results with their proof and some of  consequences can be found in Section 3.

\section{Definitions and Notation}

Denote by $
\mathbb{N}
_{+}$ the set of the positive integers, $\mathbb{N}:=\mathbb{N}_{+}\cup \{0\}.$ Let $m:=(m_{0,}$ $m_{1},...)$ be a sequence of the positive
integers not less than 2. Denote by 
\begin{equation*}
Z_{m_{k}}:=\{0,1,...,m_{k}-1\}
\end{equation*}
the additive group of integers modulo $m_{k}$.

Define the group $G_{m}$ as the complete direct product of the groups $%
Z_{m_{i}}$ with the product of the discrete topologies of $Z_{m_{j}}`$s.

The direct product $\mu $ of the measures 
$
\mu _{k}\left( \{j\}\right) :=1/m_{k}\text{ \ \ \ }(j\in Z_{m_{k}})
$
is the Haar measure on $G_{m_{\text{ }}}$with $\mu \left( G_{m}\right) =1.$

In this paper we discuss bounded Vilenkin groups,\textbf{\ }i.e. the case
when $\sup_{n}m_{n}<\infty .$

The elements of $G_{m}$ are represented by sequences 
\begin{equation*}
x:=\left( x_{0},x_{1},...,x_{j},...\right) ,\ \left( x_{j}\in
Z_{m_{j}}\right) .
\end{equation*}

Set $e_{n}:=\left( 0,...,0,1,0,...\right) \in G_m,$ the $n$-th coordinate of
which is 1 and the rest are zeros $\left( n\in \mathbb{N}\right) .$
It is easy to give a basis for the neighborhoods of $G_{m}:$ 
\begin{equation*}
I_{0}\left( x\right) :=G_{m},\text{ \ }I_{n}(x):=\{y\in G_{m}\mid
y_{0}=x_{0},...,y_{n-1}=x_{n-1}\},
\end{equation*}%
where $x\in G_{m},$ $n\in \mathbb{N}.$


If we define the so-called generalized number system based on $m$
in the following way : 
\begin{equation*}
M_{0}:=1,\ M_{k+1}:=m_{k}M_{k}\,\,\,\ \ (k\in 
\mathbb{N}
),
\end{equation*}%
then every $n\in 
\mathbb{N}
$ can be uniquely expressed as $n=\sum_{j=0}^{\infty }n_{j}M_{j},$ where $%
n_{j}\in Z_{m_{j}}$ $(j\in 
\mathbb{N}
_{+})$ and only a finite number of $n_{j}`$s differ from zero.

We introduce on $G_{m}$ an orthonormal system which is called the Vilenkin
system. At first, we define the complex-valued function $r_{k}\left(
x\right) :G_{m}\rightarrow 
\mathbb{C}
,$ the generalized Rademacher functions, by%
\begin{equation*}
r_{k}\left( x\right) :=\exp \left( 2\pi ix_{k}/m_{k}\right) ,\text{ }\left(
i^{2}=-1,x\in G_{m},\text{ }k\in 
\mathbb{N}
\right) .
\end{equation*}

Next, we define the Vilenkin system$\,\,\,\psi :=(\psi _{n}:n\in 
\mathbb{N}
)$ on $G_{m}$ by: 
\begin{equation*}
\psi _{n}(x):=\prod\limits_{k=0}^{\infty }r_{k}^{n_{k}}\left( x\right)
,\,\,\ \ \,\left( n\in 
\mathbb{N}
\right) .
\end{equation*}

Specifically, we call this system the Walsh-Paley system when $m\equiv 2.$

The norms (or quasi-norms) of the spaces $L_{p}(G_{m})$ and $%
weak-L_{p}\left( G_{m}\right) $ $\left( 0<p<\infty \right) $ are
respectively defined by 
\begin{equation*}
\left\Vert f\right\Vert _{p}^{p}:=\int_{G_{m}}\left\vert f\right\vert
^{p}d\mu ,\text{ }\left\Vert f\right\Vert _{weak-L_{p}}^{p}:=\underset{%
\lambda >0}{\sup }\lambda ^{p}\mu \left( f>\lambda \right) <+\infty .
\end{equation*}%

The Vilenkin system is orthonormal and complete in $L_{2}\left( G_{m}\right) 
$ (see \cite{Vi}).

Now, we introduce analogues of the usual definitions in Fourier-analysis. If 
$f\in L_{1}\left( G_{m}\right) $ we can define Fourier coefficients, partial
sums and Dirichlet kernels with respect to the Vilenkin system in the usual manner: 
\begin{equation*}
\widehat{f}\left( n\right) :=\int_{G_{m}}f\overline{\psi }_{n}d\mu,\ \ \ \
\ \ \
S_{n}f:=\sum_{k=0}^{n-1}\widehat{f}\left( k\right) \psi _{k},\text{ \ \ }%
D_{n}:=\sum_{k=0}^{n-1}\psi _{k\text{ }},\text{ \ \ }\left( n\in 
\mathbb{N}_{+}\right).
\end{equation*}

Let $\{q_{k}:k\geq 0\}$ be a sequence of non-negative numbers. The $n$-th  $T$ means $T_n$ for a Fourier series of $f$ are defined by
\begin{equation} \label{nor}
T_nf:=\frac{1}{Q_n}\overset{n-1}{\underset{k=0}{\sum }}q_{k}S_kf, \ \ \ \text{where} \ \ \  Q_{n}:=\sum_{k=0}^{n-1}q_{k}.
\end{equation}

We always assume that $\{q_k:k\geq 0\}$ is a sequence of non-negative numbers and $q_0>0.$ Then the summability method (\ref{nor}) generated by $\{q_k:k\geq 0\}$ is regular if and only if $\lim_{n\rightarrow\infty}Q_n=\infty.$

Let $\{q_{k}:k\geq 0\}$ be a sequence of nonnegative numbers. The $n$-th Nörlund mean $t_{n}$ for a Fourier series of $f$ \ is defined by 
\begin{equation} \label{nor0}
t_{n}f=\frac{1}{Q_{n}}\overset{n}{\underset{k=1}{\sum }}q_{n-k}S_{k}f, \ \ \ \text{where } \ \  \ Q_{n}:=\sum_{k=0}^{n-1}q_{k}.
\end{equation}%

If $q_k\equiv 1$ in \eqref{nor} and \eqref{nor0} we respectively define the Fejér means $\sigma _{n}$ and
Kernels $K_{n}$ as follows: 
\begin{equation*}
\sigma _{n}f:=\frac{1}{n}\sum_{k=1}^{n}S_{k}f\,,\text{ \ \ }K_{n}:=\frac{1}{n%
}\sum_{k=1}^{n}D_{k}.
\end{equation*}

The well-known example of N\"orlund summability is the so-called $\left(C,\alpha\right)$ mean (Ces\`aro means) for $0<\alpha<1,$ which are defined by
\begin{equation*}
\sigma_n^{\alpha}f:=\frac{1}{A_n^{\alpha}}\overset{n}{\underset{k=1}{\sum}}A_{n-k}^{\alpha-1}S_kf, 
\end{equation*}
where 
$$A_0^{\alpha}:=0,\qquad A_n^{\alpha}:=\frac{\left(\alpha+1\right)...\left(\alpha+n\right)}{n!}.$$

We also consider the "inverse" $\left(C,\alpha\right)$ means, which is an example of $T$ means:
\begin{equation*}
U_n^{\alpha}f:=\frac{1}{A_n^{\alpha}}\overset{n-1}{\underset{k=0}{\sum}}A_{k}^{\alpha-1}S_kf, \qquad 0<\alpha<1.
\end{equation*}

Let $V_n^{\alpha}$ denote
the $T$ mean, where $	\left\{q_0=0, \  q_k=k^{\alpha-1}:k\in \mathbb{N}_+\right\} ,$
that is 
\begin{equation*}
V_n^{\alpha}f:=\frac{1}{Q_n}\overset{n-1}{\underset{k=1}{\sum }}k^{\alpha-1}S_kf,\qquad 0<\alpha<1.
\end{equation*}

The $n$-th Riesz logarithmic mean $R_{n}$ and the Nörlund logarithmic mean
$L_{n}$ are defined by
\begin{equation*}
R_{n}f:=\frac{1}{l_{n}}\sum_{k=1}^{n-1}\frac{S_{k}f}{k}\text{ \ \ \ and  \ \ \ }
L_{n}f:=\frac{1}{l_{n}}\sum_{k=1}^{n-1}\frac{S_{k}f}{n-k},
\end{equation*}%
\ respectively, where $l_{n}:=\sum_{k=1}^{n-1}1/k.$

If $\{q_k:k\in\mathbb{N}\}$ is monotone and bounded sequence,
then we get the class $B_n$ of $T$ means with non-decreasing coefficients:
\begin{equation*}
B_nf:=\frac{1}{Q_n}
\sum_{k=1}^{n-1}q_kS_kf.
\end{equation*}

The $\sigma $-algebra generated by the intervals $\left\{ I_{n}\left(
x\right) :x\in G_{m}\right\} $ will be denoted by 
$\digamma _{n}\left( n\in 
\mathbb{N}\right) .$ Denote by 
$f=\left( f^{\left( n\right) },n\in \mathbb{N}\right) $ 
a martingale with respect to 
$\digamma _{n}\left( n\in \mathbb{N}\right) .$ (for details see e.g. \cite{We1}).
The maximal function of a martingale $f$ \ is defined by 
$
f^{\ast }:=\sup_{n\in \mathbb{N}}\left\vert f^{(n)}\right\vert .
$
For $0<p<\infty $ \ the Hardy martingale spaces $H_{p}$ consist of all
martingales $f$ for which 
\begin{equation*}
\left\Vert f\right\Vert _{H_{p}}:=\left\Vert f^{\ast }\right\Vert
_{p}<\infty .
\end{equation*}

If $f=\left( f^{\left( n\right) },n\in \mathbb{N}\right) $ is a martingale, then the Vilenkin-Fourier coefficients must be
defined in a slightly different manner: 
\begin{equation*}
\widehat{f}\left( i\right) :=\lim_{k\rightarrow \infty
}\int_{G_{m}}f^{\left( k\right) }\overline{\psi}_id\mu .
\end{equation*}

A bounded measurable function $a$ is called a p-atom, if there exists an interval $I$, such that
\begin{equation*}
\int_{I}ad\mu =0,\text{ \ \ }\left\Vert a\right\Vert _{\infty }\leq \mu
\left( I\right) ^{-1/p},\text{ \ \ supp}\left( a\right) \subset I.
\end{equation*}

We need the following auxiliary Lemmas:

\begin{Proposition}[see e.g. \protect\cite{We3}]
	\label{lemma2.1} A martingale $f=\left( f^{\left( n\right) },n\in \mathbb{N}\right) $ is in $H_{p}\left( 0<p\leq 1\right) $ if and only if there exists
	a sequence $\left( a_{k},k\in 
	\mathbb{N}
	\right) $ of p-atoms and a sequence $\left( \mu _{k},k\in \mathbb{N}
	\right) $ of real numbers such that, for every $n\in \mathbb{N},$ 
	\begin{equation} \label{1}
	\qquad \sum_{k=0}^{\infty }\mu _{k}S_{M_{n}}a_{k}=f^{\left( n\right) },\text{
		\ \ a.e.,}   \ \ \ \ \text{where} \ \ \ \ \sum_{k=0}^{\infty }\left\vert \mu _{k}\right\vert ^{p}<\infty . 
	\end{equation}
	Moreover,
	$$
	\left\Vert f\right\Vert _{H_{p}}\backsim \inf \left( \sum_{k=0}^{\infty
	}\left\vert \mu _{k}\right\vert ^{p}\right) ^{1/p} ,
	$$
	where the infimum is taken over all decompositions of $f$ \textit{%
		of the form} (\ref{1}).
\end{Proposition}

\section{The Main Results and Applications}

Our first main result reads:

\begin{theorem}\label{Thtmeansdiv2} 	a) Let sequence $\{q_{k}:k\geq 0\}$ is nondecreasing, satisfying condition
	\begin{equation} \label{kachzcond1}
	\frac{q_{0}}{Q_{M_{2n_{k}}+2}}\geq \frac{c}{M_{2n_{k}}}, \ \ \ \text{for some constant} \ \ c\ \ \ \text{and}\ \ \ n\in\mathbf{N}.
	\end{equation}	
	or let sequence $\{q_{k}:k\geq 0\}$ is nonincreasing, satisfying condition
	\begin{equation} \label{kachzcond2} 
	\frac{q_{M_{2n_{k}}+1}}{Q_{M_{2n_{k}}+2}}\geq \frac{c}{M_{2n_{k}}}, \ \ \ \text{for some constant} \ \ c\ \ \ \text{and}\ \ \ n\in\mathbf{N}.
	\end{equation}	
	Then for any increasing function $\varphi : \mathbf{N}_{+}\rightarrow [1,$ $\infty )$ 
	satisfying the conditions $$\lim_{n\to\infty}\varphi(n)=\infty$$ and	
	\begin{equation}\label{conddato1}
	\overline{\lim_{n\rightarrow \infty }}\frac{\log ^{2}\left( n+1\right) }{%
		\varphi \left( n+1\right) }=+\infty .
	\end{equation}
	Then there exists a martingale $f\in H_{1/2},$ such that
	\begin{equation*}
	\left\| \sup_{n\in \mathbf{N}}\frac{\vert T_{n}f\vert}{\varphi \left( n\right) }\right\|
	_{1/2}=\infty .
	\end{equation*}
	
	b) Let $0<p<1/2$ and sequence $\{q_{k}:k\geq 0\}$ is nondecreasing, or let sequence $q_k$ is nonincreasing, satisfying condition \eqref{kachzcond2}.
	Then for any increasing function $\varphi : \mathbf{N}_{+}\rightarrow [1,$ $\infty )$ 
	satisfying the condition	
	\begin{equation}\label{conddato2}
	\overline{\lim_{n\rightarrow \infty }}\frac{\left( n+1\right)^{1/p-2} }{%
		\varphi \left( n+1\right) }=+\infty,
	\end{equation}
	there exists a martingale $f\in H_{p},$ such that
	\begin{equation*}
	\left\| \sup_{n\in \mathbf{N}}\frac{\vert T_{n}f\vert }{\varphi \left( n\right) }\right\|_{\text{weak}-L_{p}}=\infty .
	\end{equation*}
\end{theorem}
\begin{proof} 
According to condition \eqref{conddato1} in case a) we conclude that there exists an increasing sequence $\left\{ n_{k}:k\in \mathbf{N}\right\} $ 
of positive integers such that
\begin{equation*}
\lim\limits_{k\rightarrow \infty }\frac{\log^2 \left(M_{2n_k+1}\right)}{\varphi \left( M_{2n_{k}+1}\right)}=+\infty.
\end{equation*}

According to condition \eqref{conddato2} we conclude that there exists an increasing sequence $\left\{ n_{k}:k\in \mathbf{N}\right\} $ 
 of positive integers such that (Here we use same indexes $n_{k}$, but it could be different)
\begin{equation*}
\lim\limits_{k\rightarrow \infty }\frac{\left(M_{2n_{k}}+2\right)^{ 1/p-2}}{
	\varphi \left( M_{2n_k}+2\right) }=+\infty, \ \ \ \text{for } \ \ \ 0<p<1/2.
\end{equation*}

Let
\begin{equation*}
f_{n_k}\left(x\right):=D_{M_{2n_k+1}}\left( x\right)-D_{M_{2n_k}}\left( x\right) .
\end{equation*}%
It is evident that

\begin{equation*}
\widehat{f}_{n_k}\left( i\right) =\left\{
\begin{array}{l}
1,\,\,\text{if\thinspace \thinspace \thinspace }
i=M_{2n_{k}},...,M_{2n_{k}+1}-1, \\
0,\,\,\text{otherwise.}
\end{array}
\right.
\end{equation*}%
and
\begin{equation}  \label{7sn}
S_{i}f_{n_{k}}\left( x\right) =\left\{
\begin{array}{ll}
D_{i}\left( x\right) -D_{M_{2n_{k}}}\left( x\right) , &
i=M_{2n_{k}}+1,...,M_{2n_{k}+1}-1, \\
f_{n_{k}}\left( x\right) , & i\geq M_{2n_{k}+1}, \\
0, & \text{otherwise.}%
\end{array}%
\right. 
\end{equation}
Since
\begin{equation*}
f_{n_k}^{\ast }\left( x\right) =\sup\limits_{n\in \mathbf{N}}\left\vert
S_{M_n}\left( f_{n_{k}};x\right) \right\vert =\left\vert f_{n_{k}}\left(
x\right) \right\vert,
\end{equation*}
we get
\begin{equation} \label{8}
\left\Vert f_{n_k}\right\Vert _{H_p}=\left\Vert f_{n_k}^{\ast
}\right\Vert_p=\left\Vert D_{M_{2n_k}}\right\Vert_p=M^{
1-1/p }_{2n_k}.  
\end{equation}

First, for case a) we consider  $p=1/2.$ By using (\ref{7sn}) and equality (see \cite{AVD})
$$D_{n}\left( x\right) =D_{M_{\left\vert n\right\vert }}\left(
x\right) +r_{\left\vert n\right\vert }\left( x\right) D_{n-M_{\left\vert n\right\vert }}\left( x\right)$$
for $1 \leq s \leq  n_k$ we get that  
\begin{eqnarray*}
	\frac{\left\vert T_{M_{2n_{k}}+M_{2s}}f_{n_k} \right\vert
	}{\varphi \left( M_{2n_{k}}+M_{2s}\right) } 
	&=&\frac{1}{\varphi \left( M_{2n_{k}}+M_{2s}\right) }\frac{1}{Q_{M_{2n_{k}}+M_{2s}}}\left\vert
	\sum\limits_{j=0}^{M_{2n_{k}}+M_{2s}-1}q_jS_{j}f_{n_{k}}
	\right\vert \\
	&=&\frac{1}{\varphi \left( M_{2n_{k}}+M_{2s}\right) }\frac{1}{Q_{M_{2n_{k}}+M_{2s}}}\left\vert
	\sum\limits_{j=M_{2n_{k}}}^{M_{2n_{k}}+M_{2s}-1}
	q_jS_{j}f_{n_{k}} \right\vert \\
	&=&\frac{1}{\varphi \left(M_{2n_{k}}+M_{2s}\right) }\frac{1}{Q_{M_{2n_{k}}+M_{2s}}}\left\vert
	\sum\limits_{j=M_{2n_{k}}}^{M_{2n_{k}}+M_{2s}-1}
	q_j\left( D_{j} -D_{M_{2n_{k}}} \right) \right\vert \\
	&=&\frac{1}{\varphi \left(M_{2n_{k}}+M_{2s}\right) }\frac{1}{Q_{M_{2n_{k}}+M_{2s}}}\left\vert
	\sum\limits_{j=0}^{M_{2n_{k}}-1}
	q_{j+M_{2n_{k}}}\left( D_{j+M_{2n_{k}}}-D_{M_{2n_{k}}} \right) \right\vert \\
	&=&\frac{1}{\varphi \left( M_{2n_{k}}+M_{2s}\right) }\frac{1}{Q_{M_{2n_{k}}+M_{2s}}}\left\vert
	\sum\limits_{j=0}^{M_{2s}-1}q_{j+M_{2m_k}}D_{j}  \right\vert 
\end{eqnarray*}

Let 
$x\in I_{2s}\setminus I_{2s+1}.$
Then
\begin{eqnarray*}
	&&\frac{\left\vert T_{M_{2n_{k}}+M_{2s}}f_{n_k}\left( x\right) \right\vert}{\varphi \left( M_{2n_{k}}+M_{2s}\right) } =\frac{1}{\varphi \left( M_{2n_{k}}+M_{2s}\right) }\frac{1}{Q_{M_{2n_{k}}+M_{2s}}}\left\vert
	\sum\limits_{j=0}^{M_{2n_{k}}-1}q_{j+M_{2n_k}}j \right\vert 
\end{eqnarray*}
Let sequence $\{q_{k}:k\geq 0\}$ is nondecreasing. Then according to condition \eqref{kachzcond1} we find that
\begin{eqnarray*}
	\frac{\left\vert T_{M_{2n_{k}}+M_{2s}}f_{n_k}\left( x\right) \right\vert}{\varphi \left( M_{2n_{k}}+M_{2s}\right) } 
	&\geq&\frac{1}{\varphi \left(M_{2n_{k}}+M_{2s}\right) }\frac{q_0}{Q_{M_{2n_{k}}+M_{2s}}}
	\sum\limits_{j=0}^{M_{2s}-1}j\\
	&\geq&\frac{1}{\varphi \left( M_{2n_{k}+1}\right) }\frac{q_0}{Q_{M_{2n_{k}}+M_{2s}}}
	\sum\limits_{j=0}^{M_{2s}-1}j \geq \frac{cM^2_{2s}}{M_{2n_{k}}\varphi \left( M_{2n_{k}+1}\right)}.
\end{eqnarray*}
Let sequence $\{q_{k}:k\geq 0\}$ is nonincreasing. Since $\varphi : \mathbf{N}_{+}\rightarrow [1,$ $\infty )$  is increasing sequence, by using condition \eqref{kachzcond2} we get that
\begin{eqnarray*}
	\frac{\left\vert T_{M_{2n_{k}}+M_{2s}}f_{n_k}\left( x\right) \right\vert}{\varphi \left( M_{2n_{k}}+M_{2s}\right) } 
	\geq\frac{1}{\varphi \left( M_{2n_{k}}+M_{2s}\right) }\frac{q_{M_{2n_{k}}+M_{2s}-1}}{Q_{M_{2n_{k}}+M_{2s}}}
	\sum\limits_{j=0}^{M_{2s}-1}j
	\geq  \frac{cM^2_{2s}}{M_{2n_{k}}\varphi \left( M_{2n_{k}+1}\right)}.
\end{eqnarray*}
Hence,
\begin{eqnarray*}
&&\underset{G_m}{\int }\left( \sup_{n\in \mathbf{N}}\frac{\vert T_{n}f_{n_k}\vert }{\varphi \left( n\right) } \right) ^{1/2}d\mu \geq \sum_{s=1}^{n_{k}}\underset{I_{2s}\setminus I_{2s+1}}{\int }\left| \frac{T_{M_{2n_{k}}+M_{2s}}f_{n_k}}{\varphi(M_{2n_{k}}+M_{2s})} \right| ^{1/2}d\mu \\ 
&\geq& \overset{n _{k}}{\underset{s=1}{\sum}}\underset{I_{2s}\setminus I_{2s+1}}{\int }\left(\frac{cM^2_{2s}}{M_{2n_{k}}\varphi \left( M_{2n_{k}+1}\right)}\right)^{1/2}d\mu 
\geq \frac{c}{\left(M_{2n_{k}}\varphi \left( M_{2n_{k}+1}\right)\right)^{1/2}}\sum_{s=1}^{n _{k}}M_{2s} \vert I_{2s}\setminus I_{2s+1} \vert  \\
&\geq&  \frac{c}{\left(M_{2n_{k}}\varphi \left( M_{2n_{k}+1}\right)\right)^{1/2}}\sum_{s=1}^{n_{k}} 1\geq  \frac{cn_{k}}{\left(M_{2n_{k}}\varphi \left( M_{2n_{k}+1}\right)\right)^{1/2}}.
\end{eqnarray*}

From (\ref{8}) we get that
\begin{eqnarray*}
\frac{\left( \int_{G_m}\left( \sup_{n\in \mathbf{N}}\frac{ \vert T_{n}f_{n_k}\vert }{\varphi \left( n\right) } \right) ^{1/2}d\mu  \right) ^{2}}{\left\| f_{n_{k}}\right\| _{H_{1/2}}} &\geq& \frac{cn^2_{k}M_{2n_{k}}}{M_{2n_{k}}\varphi \left( M_{2n_{k}+1}\right)}\geq \frac{cn^2_{k}}{\varphi \left( M_{2n_{k}+1}\right)}\\
&\geq& \frac{c{\left(2n_k+1\right)}^2}{\varphi \left( M_{2n_{k}+1}\right)}
\geq \frac{c\log^2 \left(M_{2n_k+1}\right)}{\varphi \left( M_{2n_{k}+1}\right)}
\to \infty, \ \ \text{as} \ \ k\to \infty.
\end{eqnarray*}
This complete proof of part a).

Next, we consider case $0<p<1/2.$ In the view of identities \eqref{7sn} of Fourier coefficients we find that
\begin{eqnarray*}
	\frac{\left\vert T_{M_{2n_{k}}+2}f_{n_k} \right\vert
	}{\varphi \left( M_{2n_{k}}+2\right) } 
	&=&\frac{1}{\varphi \left( M_{2n_{k}}+2\right) }\frac{1}{Q_{M_{2n_{k}}+2}}\left\vert
	\sum\limits_{j=0}^{M_{2n_{k}}+1}q_jS_{j}f_{n_{k}}
	\right\vert \\
	&=&\frac{1}{\varphi \left(M_{2n_{k}}+2\right) }\frac{1}{Q_{M_{2n_{k}}+2}}\left\vert
	q_{M_{2n_{k}}+1}\left( D_{M_{2n_{k}}+1}-D_{M_{2n_{k}}} \right) \right\vert \\
	&=&\frac{1}{\varphi \left(M_{2n_{k}}+2\right) }\frac{1}{Q_{M_{2n_{k}}+2}}\left\vert
	q_{M_{2n_{k}}+1}\psi_{M_{2n_{k}}}  \right\vert \\
	&=&\frac{1}{\varphi \left(M_{2n_{k}}+2\right) }\frac{q_{M_{2n_{k}}+1}}{Q_{M_{2n_{k}}+2}}
\end{eqnarray*}

Let sequence $\{q_{k}:k\geq 0\}$ is nondecreasing. Then
\begin{eqnarray*}
\frac{\left\vert T_{M_{2n_{k}}+2}f\left( x\right) \right\vert}{\varphi \left( M_{2n_{k}}+2\right) }\geq
\frac{1}{\varphi \left(M_{2n_{k}}+2\right) }\frac{q_{M_{2n_{k}}+1}}{q_{M_{2n_{k}}+1}\left(M_{2n_{k}}+2\right)}
\geq \frac{c}{M_{2n_{k}}\varphi \left(  M_{2n_{k}}+2\right) }
\end{eqnarray*}
Let sequence $\{q_{k}:k\geq 0\}$ is nonincreasing. Then, according condition \eqref{kachzcond2}  we find that
\begin{eqnarray*}
	\frac{\left\vert T_{M_{2n_{k}}+2}f\left( x\right) \right\vert
	}{\varphi \left( M_{2n_{k}}+2\right) } 
	=\frac{1}{\varphi \left(M_{2n_{k}}+2\right) }\frac{q_{M_{2n_{k}}+1}}{Q_{M_{2n_{k}}+2}}\geq \frac{c}{M_{2n_{k}}\varphi \left(  M_{2n_{k}}+2\right) }
\end{eqnarray*}
Hence,
\begin{eqnarray*}
	&&\mu \left\{ x\in G_m:\frac{\left\vert T_{M_{2n_{k}}+2}f\left(x\right) \right\vert }{\varphi \left(  M_{2n_{k}}+2\right)}
	\geq\frac{c}{M_{2n_{k}}\varphi\left(M_{2n_{k}}+2\right)}\right\}=\vert G_m\vert=1.
\end{eqnarray*}
Then from (\ref{8}) we get that
\begin{eqnarray*}
	&&\frac{\frac{c}{M_{2n_{k}}\varphi \left(M_{2n_{k}}+2\right)}\left\{ \mu
		\left\{ x\in G_m:\frac{\left\vert T_{ M_{2n_{k}}+2}f_{n_{k}}\left(x\right) \right\vert }{\varphi \left(  M_{2n_{k}}+2\right) }\geq \frac{c}{M_{2n_{k}}\varphi \left(  M_{2n_{k}}+2\right) }\right\} \right\} ^{1/p}}{
		\left\Vert f_{n_{k}}\right\Vert _{H_{p}}} \\
	&\geq &\frac{c}{M_{2n_{k}}\varphi \left( M_{n_{k}}+2\right) M_{2n_{k}}^{1-1/p}} =\frac{cM^{1/p-2}_{2n_{k} }}{\varphi\left( M_{2n_{k}}+2\right) }\\
	&\geq& \frac{c\left(M_{2n_{k} }+2\right)^{1/p-2}}{\varphi\left( M_{2n_{k}}+2\right) }
	\rightarrow \infty, \ \ \ \text{as }\ \ \ k\rightarrow
	\infty .
\end{eqnarray*}

The proof is complete.
\end{proof}

As application we get well-known result for the weighted maximal operator of Fejér means which was considered in Tephnadze \cite{tep2,tep3}:

\begin{corollary}
	Let $\varphi : \mathbf{N}_{+}\rightarrow [1,\infty )$  be any increasing
	function satisfying the conditions 
	$$\lim_{n\to\infty}\varphi(n)=\infty$$
	and
	\begin{equation}\label{cond2}
	\overline{\lim_{n\rightarrow \infty }}\frac{\left( n+1\right) ^{1/p-2}\log ^{2\left[ 1/2+p\right]}\left( n+1\right)}{\varphi(n)}=+\infty .  
	\end{equation}%
	Then
	\begin{equation*}
\frac{\left\Vert \sup_{n\in \mathbf{N}}\frac{	\vert\sigma_{n}f\vert}{
			\varphi {(n)}}\right\Vert_{1/2}}{\left\Vert
		f\right\Vert_{H_{1/2}}}=\infty 
	\end{equation*}	
and
	\begin{equation*}
\frac{\left\Vert \sup_{n\in \mathbf{N}}\frac{	\vert\sigma_{n}f\vert}{
\varphi {(n)}}\right\Vert_{{\text{weak}-L_p}}}{\left\Vert
		f\right\Vert_{H_p}}=\infty.
	\end{equation*}
\end{corollary}

We also present some new results on $T$ means with respect to Vilenkin systems which follows Theorem \ref{Thtmeansdiv2}:

\begin{corollary}
Theorem \ref{Thtmeansdiv2}  holds  true for $U_n^{\alpha}f,$	$V_n^{\alpha}f$ and $B_nf$ means with respect to Vilenkin systems.
\end{corollary}

\end{document}